\documentclass{amsart}
\usepackage{mathrsfs}
\usepackage{amssymb}
\usepackage{graphicx}
\usepackage{stmaryrd}
\usepackage{enumitem} 
\usepackage{color}
\usepackage{comment}

\usepackage{pgfplots}
\usepackage{tikz-cd}
\usetikzlibrary{decorations.pathmorphing}
\usetikzlibrary{external}

\DeclareMathAlphabet{\mathbbe}{U}{bbold}{m}{n}

\ifx\color@rgb\@undefined\else
  \definecolor{cyan}{rgb}{.15,.38,.61}
\fi

\ifx\color@rgb\@undefined\else
  \definecolor{maroon}{rgb}{0.5,0,0}
\fi

\ifx\color@rgb\@undefined\else

  \definecolor{lilac}{rgb}{0.45,0.31,0.59}
\fi

\newtheorem{theorem}{Theorem}[subsection]

\newtheorem{proposition}[theorem]{Proposition}

\theoremstyle{definition}
\newtheorem{definition}[theorem]{Definition}
\newtheorem{example}[theorem]{Example}
\newtheorem{notation}[theorem]{Notation}
\newtheorem*{acknowledgements}{Acknowledgments}

\theoremstyle{remark}
\newtheorem{remark}[theorem]{Remark}
\newtheorem{exercise}[theorem]{Exercise}

\makeatletter
\@addtoreset{section}{part}
\makeatother  

\makeatletter
\let\c@equation\c@theorem
\makeatother
\numberwithin{equation}{subsection}

\newcommand{\ob}{\mathrm{ob}}
\newcommand{\To}{\Rightarrow}

\newcommand{\cat}[1]{\textup{\textsf{#1}}}
\newcommand{\vcat}[1]{{{#1}\text{-}\cat{Cat}}}
\newcommand{\mr}[1]{\mathrm{#1}}

\newcommand{\cC}{\mathsf{C}}

\newcommand{\DDelta}{\mathbbe{\Delta}}

\begin{document}

\title{Complicial sets, an overture}
\author{Emily Riehl}
\date{\today}

\address{Department of Mathematics\\Johns Hopkins University \\ 3400 N Charles Street \\ Baltimore, MD 21218}
\email{eriehl@math.jhu.edu}

\begin{abstract}
The aim of these notes is to introduce the intuition motivating the notion of a \emph{complicial set}, a simplicial set with certain marked ``thin'' simplices that witness a composition relation between the simplices on their boundary. By varying the marking conventions, complicial sets can be used to model $(\infty,n)$-categories for each $n \geq 0$, including $n=\infty$. For this reason, complicial sets present a fertile setting for thinking about weak infinite dimensional categories in varying dimensions. This overture is presented in three acts: the first introducing simplicial models of higher categories; the second defining the Street nerve, which embeds strict $\omega$-categories as \emph{strict} complicial sets; and the third exploring an important saturation condition on the marked simplices in a complicial set and presenting  a variety of model structures that capture their basic homotopy theory. Scattered throughout are suggested exercises for the reader who wants to engage more deeply with these notions.
\end{abstract}

\maketitle

\setcounter{tocdepth}{2}
\tableofcontents

As the objects that mathematicians study increase in sophistication, so do their natural habitats. On account of this trend, it is increasingly desirable to replace mere 1-categories of objects and the morphisms between them, with infinite-dimensional categories containing 2-morphisms between 1-morphisms, 3-morphisms between 2-morphisms, and so on. The principle challenge in working with infinite-dimensional categories is that the naturally occurring examples are \emph{weak} rather than \emph{strict}, with composition of $n$-morphisms only associative and unital up to an $n+1$-morphism that is an ``equivalence'' in some sense. The complexity is somewhat reduced in the case of $(\infty,n)$-\emph{categories}, in which all $k$-morphisms are weakly invertible for $k>n$, but even in this case, explicit models of these schematically defined $(\infty,n)$-categories can be extremely complicated.

\emph{Complicial sets} provide a relatively parsimonious model of infinite-dimensional categories, with special cases modeling $(\infty,0)$-categories (also called $\infty$-\emph{groupoids}), $(\infty,1)$-categories (the ubiquitous $\infty$-\emph{categories}), indeed $(\infty,n)$-categories for any $n$, and also including the general case of $(\infty,\infty)$-categories. Unlike other models of infinite-dimensional categories, the definition of a complicial set is extremely simple to state: it is a simplicial set with a specified collection of marked ``thin'' simplices, in which certain elementary anodyne extensions exist. These anodyne extensions provide witnesses for a weak composition law and guarantee that the thin simplices are equivalences in a sense defined by this weak composition. 

This overture is dividing into three acts, each comprising one part of the three-hour mini course that generated these lecture notes. In the first, we explore how a simplicial set can be used to model the weak composition of an $(\infty,1)$-category and consider the extra structure required to extend these ideas to provide a simplicial model of $(\infty,2)$-categories. This line of inquiry leads naturally to the definition of a a complicial set as a \emph{stratified} (read ``marked'') simplicial set in which composable simplices admit composites.

 In the second part, we delve into the historical motivations for this stratified simplicial set based model of higher categories. John Roberts proposed the original definition of \emph{strict} complicial sets, which admit unique extensions along the elementary anodyne inclusions, as a conjectural model for strict $\omega$-categories \cite{roberts}. Ross Street defined a nerve functor from $\omega$-categories into simplicial sets \cite{street}, and Dominic Verity proved that it defines a full and faithful embedding into the category of stratified simplicial sets whose essential image is precisely the strict complicial sets \cite{verity-complicial}. While we do not have the space to dive into proof of this result here, we nonetheless describe the Street nerve in some detail as it is an important source of examples of both strict and also weak complicial sets, as is explained in part three.

In the final act, we turn our attention to those complicial sets that most accurately model $(\infty,n)$-categories. Their markings are \emph{saturated}, in the sense that every simplex that behaves structurally like an equivalence, is marked. We present a variety of model structures, due to Verity, that encode the basic homotopy theory of complicial sets of various flavors, including those that are $n$-\emph{trivial}, with every simplex above dimension $n$ marked, and saturated. The saturation condition is essential for a conjectural equivalence between the complicial sets models of $(\infty,n)$-categories and other models known to satisfy the axiomatization of Barwick--Schommer-Pries \cite{BSP}, which passes through a complicial nerve functor due to Verity. This result will appear in a future paper.

\begin{acknowledgements}
This document evolved from lecture notes written to accompany a three-hour mini course entitled ``Weak Complicial Sets'' delivered at the Higher Structures in Geometry and Physics workshop at the MATRIX Institute from June 6-7, 2016. The author wishes to thank Marcy Robertson and Philip Hackney, who organized the workshop, the MATRIX Institute  for providing her with the opportunity to speak about this topic, and the NSF for financial support through the grant DMS-1551129. In addition, the author is grateful for personal conversations with the two world experts---Dominic Verity and Ross Street---who she consulted while preparing these notes.
\end{acknowledgements}

\section{Introducing complicial sets}\label{part:introduction}

Infinite dimensional categories have morphisms in each dimension that satisfy a weak composition law, which is associative and unital up to higher-dimensional morphisms rather than on the nose. There is no universally satisfactory definition of ``weak composition''; instead a variety of models of infinite-dimensional categories provide settings to work with this notion.

A \emph{complicial set}, nee.~\emph{weak complicial set}, is a \emph{stratified} simplicial set, with a designed subset of ``thin'' marked simplices marked, that admits extensions along certain maps. Complicial sets model weak infinite-dimensional categories, sometimes called $(\infty,\infty)$-\emph{categories}. By requiring all simplices above a fixed dimension to be thin, they can also model $(\infty,n)$-categories for all $n \in [0,\infty]$.

\emph{Strict} complicial sets were first defined by  Roberts \cite{roberts} with the intention of constructing a simplicial model of strict $\omega$-categories. He conjectured that it should be possible to extend the classical nerve to define an equivalence from the category of strict $\omega$-categories to the category of strict complicial sets.  Street defined this nerve \cite{street}, providing a fully precise statement of what is known as the Street--Roberts conjecture, appearing as Theorem \ref{thm:street-roberts}. Verity proved the Street--Roberts conjecture \cite{verity-complicial} and then subsequently defined and developed the theory of the weak variety complicial sets  \cite{verity-weak-i, verity-weak-ii} that is the focus here.

We begin in \S\ref{sec:quasi} by revisiting how a quasi-category (an unmarked simplicial set) models an $(\infty,1)$-category. This discussion enables us to  explore what would be needed to model an $(\infty,2)$-category as a simplicial set  in \S\ref{sec:level-2}. These excursions motivate the definition of stratified simplicial sets in \S\ref{sec:strat} and then complicial sets in \S\ref{sec:complicial}. We conclude in \S\ref{sec:n-trivial} by defining $n$-\emph{trivial} complicial sets, which, like  $(\infty,n)$-categories, have non-invertible simplices concentrated in low dimensions.

We assume the reader has some basic familiarity with the combinatorics of simplicial sets and adopt relatively standard notations, e.g., $\Delta[n]$ for the standard $n$-simplex and $\Lambda^k[n]$ for the horn formed by those faces that contain the $k$th vertex.

\subsection{Quasi-categories as $(\infty,1)$-categories}\label{sec:quasi}

The most popular model for $(\infty,1)$-categories were first introduced by Michael Boardman and Rainer Vogt under the name \emph{weak Kan complexes} \cite{BV}.

\begin{definition} A \textbf{quasi-category} is a simplicial set $A$ so that every inner horn admits a filler
\[
  \raisebox{-0.5\height}{\includegraphics{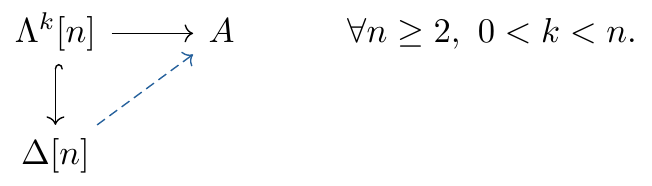}}
\]
\end{definition}

This presents an $(\infty,1)$-category with:
\begin{itemize}
\item $A_0$ as the set of objects; 
\item $A_1$ as the set of 1-cells with sources and targets determined by the face maps
\[ 
\raisebox{-0.5\height}{\includegraphics{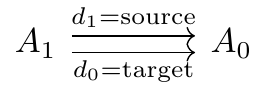}}
\]
and degenerate 1-simplices serving as identities; 
\item $A_2$ as the set of 2-cells; 
\item $A_3$ as the set of 3-cells, and so on.
\end{itemize}

The weak 1-category structure arises as follows.  A 2-simplex
\begin{equation}\label{eq:gen-2-simplex}
\raisebox{-0.5\height}{\includegraphics{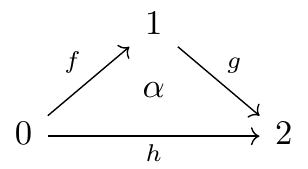}}
\end{equation}
provides a witness that $h \simeq gf$. 

\begin{notation}
We adopt the convention throughout of always labeling the vertices of an $n$-simplex by $0,\ldots, n$ to help orient each picture. This notation does not assert that the vertices are necessarily distinct.
\end{notation}

A 3-simplex then
\[
\raisebox{-0.5\height}{\includegraphics{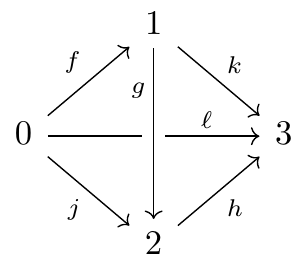}}
\]
provides witnesses that $h(gf) \simeq hj \simeq \ell \simeq kf \simeq (hg)f$.

The \textbf{homotopy category} of a quasi-category $A$ is the category whose objects are vertices and whose morphisms are a quotient of $A_1$ modulo the relation $f \simeq g$ that identifies a pair of parallel edges if and only if there exists a 2-simplex.
\begin{equation}\label{eq:f-sim-g}
\raisebox{-0.5\height}{\includegraphics{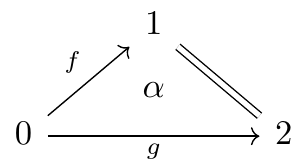}}
\end{equation}

\begin{notation}
Here and elsewhere the notation ``$=$'' is used for degenerate simplices.
\end{notation}

\begin{exercise}Formulate alternate versions of the relation $f \simeq g$ and prove that in a quasi-category each of these relations defines an equivalence relation and furthermore that these relations are all equivalent.
\end{exercise}

The composition operation witnessed by 2-simplices is not unique on the nose but it is unique up to the notion of homotopy just introduced.

\begin{exercise} Prove that the homotopy category is a strict 1-category.
\end{exercise}

A quasi-category is understood as presenting an $(\infty,1)$-category rather than an $(\infty,\infty)$-category because each 2-simplex is invertible up to a 3-simplex, and each 3-simplex is invertible up to a 4-simplex, and so on, in a sense we  now illustrate. First consider a 2-simplex as in \eqref{eq:f-sim-g}. This data can be used to define a horn $\Lambda^1[3] \to A$ whose other two faces are degenerate
\[
\raisebox{-0.5\height}{\includegraphics{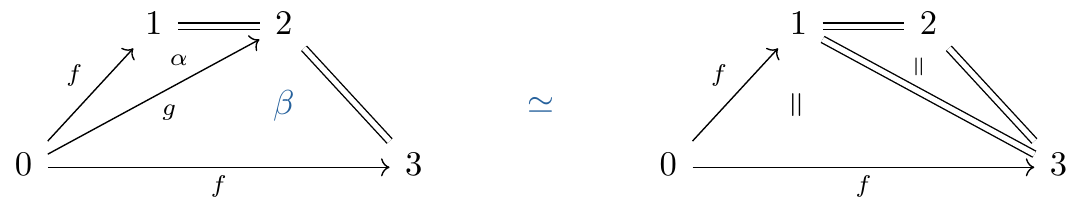}}
\]
which can be filled to define a ``right inverse'' $\beta$ in the sense that this pair of 2-cells bound a 3-simplex with other faces degenerate. Similarly, there is a ``special outer horn''\footnote{Special outer horns $\Lambda^0[n] \to A$ and $\Lambda^n[n]\to A$ have first or last edges mapping to 1-\emph{equivalences} (such as degeneracies) in $A$, as introduced in Definition \ref{defn:1-equivalence} below.} $\Lambda^3[3] \to A$
\[
\raisebox{-0.5\height}{\includegraphics{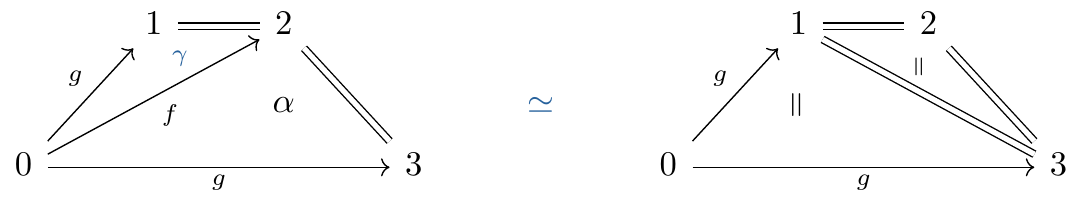}}
\]
which can be filled to define a ``left inverse'' $\gamma$. In this sense, $\alpha$ is an \emph{equivalence} up to 3-simplices, admitting left and right inverses along the boundary of a pair of three simplices.

This demonstrates that 2-simplices with a degenerate outer edge admit left and right inverses, but what if $\alpha$ has the form \eqref{eq:gen-2-simplex}? In this case, we can define a $\Lambda^1[3] \to A$ horn
\[
\raisebox{-0.5\height}{\includegraphics{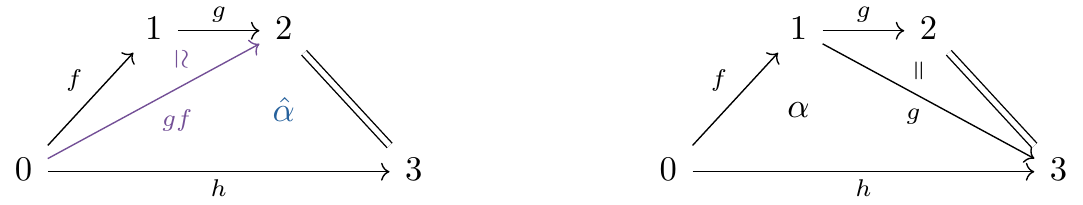}}
\]
whose 3rd face is constructed by filling a horn $\Lambda^1[2] \to A$. In this sense, any 2-simplex is equivalent to one with last (or dually first) edge degenerate. 

\begin{exercise} Generalize this argument to show that the higher-dimensional simplices in a quasi-category are also weakly invertible.
\end{exercise}

\subsection{Towards a simplicial model of $(\infty,2)$-categories}\label{sec:level-2}

Having seen how a simplicial set may be used to model an $(\infty,1)$-category, it is natural to ask how a simplicial set might model an $(\infty,2)$-category. A reasonable idea would be interpret the 2-simplices as inhabited by not necessarily invertible 2-cells pointing in a consistent direction. The problem with this is 
 that the 2-simplices need to play a dual role: they must also witness composition of 1-simplices, in which case it does not make sense to think of them as inhabited by non-invertible cells. The idea is to mark as ``thin'' the witnesses for composition and then demand that these marked 2-simplices behave as 2-dimensional equivalences in a sense that can be intuited from the preceding three diagrams.
 
Then 3-simplices can be thought of as witnesses for composition of not-necessarily thin 2-simplices. For instance, given a pair of 2-simplices $\alpha$ and $\beta$ with boundary as displayed below, the idea is to build a $\Lambda^2[3]$-horn
\[
\raisebox{-0.5\height}{\includegraphics{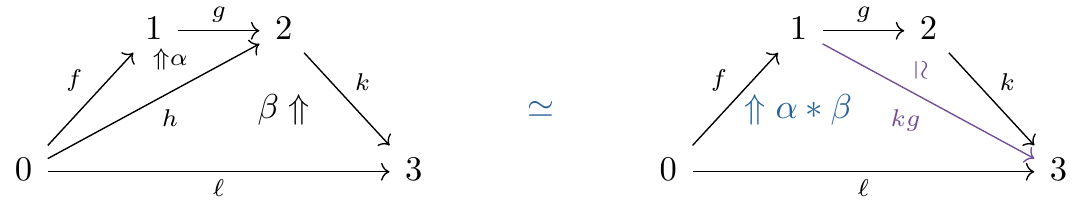}}
\]
whose 0th face is a thin filler of the $\Lambda^1[2]$-horn formed by $g$ and $k$. The 2nd face, defined by filling the horn $\Lambda^2[3]$-horn, defines a composite 2-simplex, which is witnessed by the (thin) 3-simplex. Note that because the 0th face is thin, its 1st edge is interpreted as a composite $kg$ of $g$ and $k$, which is needed so that the boundary of the new 2-cell agrees with the boundary of the pasted composite of $\beta$ and $\alpha$. Since the 3-simplex should be thought of as a witness to a composition relation involving the 2-simplices that make up its boundary, the three simplex should also be regarded as ``thin.''

A similar $\Lambda^1[3]$-horn  can be used to define composites where the domain of $\alpha$ is the last, rather than the first, edge of the codomain of $\beta$. It is in this way that simplicial sets with certain marked simplices are used to model $(\infty,2)$-categories or indeed $(\infty,n)$-categories for any $n \in [0,\infty]$. We now formally introduce \emph{stratified simplicial sets} before stating the axioms that define these \emph{complicial sets}.

\subsection{Stratified simplicial sets}\label{sec:strat}

We have seen that for a simplicial set to model an infinite-dimensional category with non-invertible morphisms in each dimension, it should have a distinguished set of ``thin'' $n$-simplices witnessing composition of $(n-1)$-simplices. Degenerate simplices are always thin in this sense. Furthermore, the intuition that the ``thin'' simplices are the equivalences, in a sense that is made precise in \S\ref{part:saturation}, suggests that certain 1-simplices might also be marked as thin. This motivates the following definition:

\begin{definition} A \textbf{stratified simplicial set} is a simplicial set with a designated subset of \textbf{marked} or \textbf{thin} positive-dimensional simplices that includes all degenerate simplices. A map of stratified sets is a simplicial map that preserves thinness.
\end{definition}

\begin{notation} The symbol ``$\simeq$'' is used throughout to decorate thin simplices.
\end{notation}

There are left and right adjoints
\[
\raisebox{-0.5\height}{\includegraphics{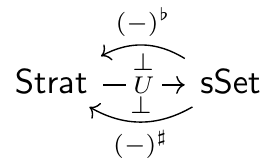}}
\]
to the forgetful functor from stratified simplicial sets to ordinary simplicial sets, both of which are full and faithful. The left adjoint assigns a simplicial set its \textbf{minimal stratification}, with only degenerate simplices marked, while the right adjoint assigns the \emph{maximal stratification}, marking all simplices. When a simplicial set is regarded as a stratified simplicial set, the default convention is to assign the minimal stratification, with the notation ``$(-)^\flat$'' typically omitted.

\begin{definition} An inclusion $U \hookrightarrow V$ of stratified simplicial sets is:
\begin{itemize}
\item \textbf{regular}, denoted $U \hookrightarrow_rV$, if thin simplices in $U$ are created in $V$ (a simplex is thin in $U$ if and only if its image in $V$ is thin); and 
\item \textbf{entire}, denoted $U\hookrightarrow_eV$, if the map is the identity on underlying simplicial sets (in which case the only difference between $U$ and $V$ is that more simplices are marked in $V$).
\end{itemize}
\end{definition}

A standard inductive argument, left to the reader, proves:

\begin{proposition}\label{prop:monos}
The monomorphisms in $\cat{Strat}$ are generated under pushout and transfinite composition by 
\[\{ \partial\Delta[n]\hookrightarrow_r\Delta[n] \mid n \geq 0\} \cup \{ \Delta[n]\hookrightarrow_e\Delta[n]_t \mid n \geq 1\},\]
where the top-dimensional $n$-simplex in $\Delta[n]_t$ is thin.
\end{proposition}

\begin{exercise} Prove this.
\end{exercise}

\subsection{Complicial sets}\label{sec:complicial}

A stratified simplicial set is a simplicial set with enough structure to talk about composition of simplices. A complicial set is a stratified simplicial set in which composites exist and in which thin witnesses to composition compose to thin simplices, an associativity condition that will also play a role in establishing their equivalence-like nature. The following form of the definition of a (weak) \emph{complicial set}, due to Verity \cite{verity-weak-i}, modifies an earlier equivalent presentation due to Street \cite{street}. Verity's modification focuses on a particular set of $k$-\emph{admissible} $n$-\emph{simplices}, thin $n$-simplices that witness that the $k$th face is a composite of the $(k+1)$th and $(k-1)$th simplices.

\begin{definition}[$k$-admissible $n$-simplex]
The $k$-\textbf{admissible} $n$-\textbf{simplex} is an entire superset of the standard $n$-simplex with certain additional faces marked thin: a non-degenerate $m$-simplex in $\Delta^k[n]$ is thin if and only if it contains the vertices $\{k-1,k,k+1\} \cap [n]$. Thin faces include:
\begin{itemize}
\item the top dimensional $n$-simplex 
\item all codimension-one faces except for the $(k-1)$th, $k$th, and $(k+1)$th
\item the 2-simplex spanned by $[k-1,k, k+1]$ in the case of inner horns or the edge $[k-1,k,k+1] \cap [n]$ in the case of outer horns.
\end{itemize}
\end{definition}

\begin{definition}\label{defn:complicial}
A \textbf{complicial set} is a stratified simplicial set that admits extensions along the \textbf{elementary anodyne extensions}, which are generated under pushout and transfinite composition by the following two sets of maps:
\begin{enumerate}
\item The \textbf{complicial horn extensions} 
\[\Lambda^k[n]\hookrightarrow_r\Delta^k[n]\quad \mathrm{for} \quad n \geq 1,\ 0 \leq k \leq n\] are regular inclusions of $k$-\textbf{admissible $n$-horns}. An inner admissible $n$-horn parametrizes ``admissible composition'' of a pair of $(n-1)$-simplices. The extension defines a composite $(n-1)$-simplex together with a thin $n$-simplex witness.
\begin{equation}\label{eq:complicial-comp}
\raisebox{-0.5\height}{\includegraphics{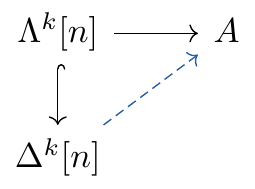}}
\end{equation}
\item The \textbf{complicial thinness extensions} \[\Delta^k[n]' \hookrightarrow_e \Delta^k[n]''\quad \mathrm{for}\quad n \geq 1,\ 0 \leq k \leq n,\] are entire inclusions of two entire supersets of $\Delta^k[n]$. The stratified simplicial set $\Delta^k[n]'$ is obtained from $\Delta^k[n]$ by also marking the $(k-1)$th and $(k+1)$th faces, while $\Delta^k[n]''$ has all codimension-one faces marked. This extension problem
\begin{equation}\label{eq:complicial-thin}
\raisebox{-0.5\height}{\includegraphics{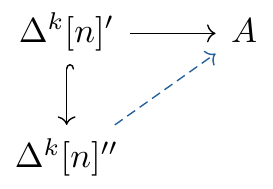}}
\end{equation}
 demands that whenever the composable pair of simplices in an admissible horn are thin, then so is any composite.
\end{enumerate}
\end{definition}

\begin{definition}
A \textbf{strict complicial set} is a stratified simplicial set that admits unique extensions along the elementary anodyne extensions \eqref{eq:complicial-comp} and \eqref{eq:complicial-thin}.
\end{definition}

\begin{example}[complicial horn extensions]
To gain familiarity with the elementary anodyne extensions, let us draw  the complicial horn extensions in low dimensions, using cyan to depict simplices present in the codomain but not the domain and ``$\simeq$'' to decorate thin simplices. The labels on the simplices are used to suggest the interpretation of certain data as composites of other data, but recall that in a (non-strict) complicial set there is no single simplex designated as \emph{the} composite of an admissible pair of simplices. Rather, the fillers for the complicial horn extensions provide \emph{a} composite and a witness to that relation.
\[
\bullet \quad \Lambda^1[2]\hookrightarrow_r\Delta^1[2]\qquad
\vcenter{
\raisebox{-0.5\height}{\includegraphics{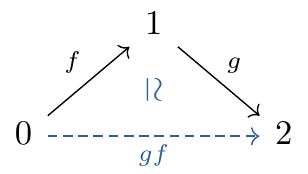}}
}
\]
\[
\bullet \quad \Lambda^0[2]\hookrightarrow_r\Delta^0[2]\qquad
\vcenter{
\raisebox{-0.5\height}{\includegraphics{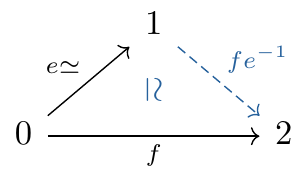}}
}
\]
$\bullet\quad\Lambda^2[3]\hookrightarrow_r\Delta^2[3]$
\[
\vcenter{
\raisebox{-0.5\height}{\includegraphics{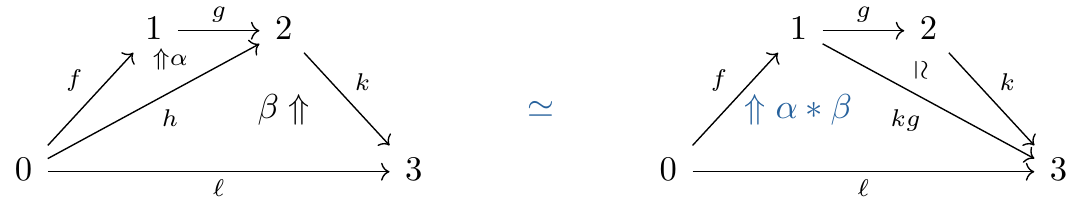}}
}
\]
$\bullet\quad\Lambda^0[3]\hookrightarrow_r\Delta^0[3]$
\[
\vcenter{
\raisebox{-0.5\height}{\includegraphics{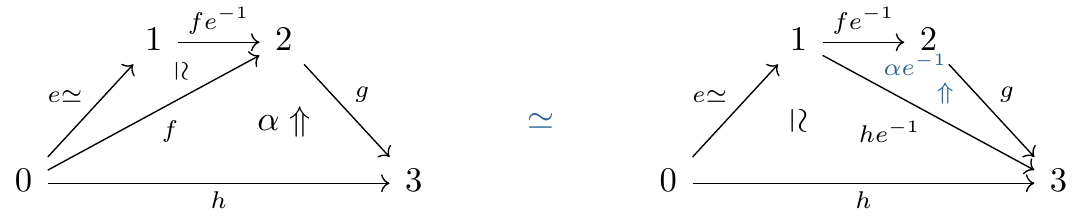}}
}
\]
$\bullet\quad$ For $\Lambda^2[4]\hookrightarrow_r\Delta^2[4]$ the non-thin codimension-one faces in the horn define the two 3-simplices with a common face displayed on the left, while their composite is a 3-simplex as displayed on the right.
\[
\raisebox{-0.5\height}{\includegraphics{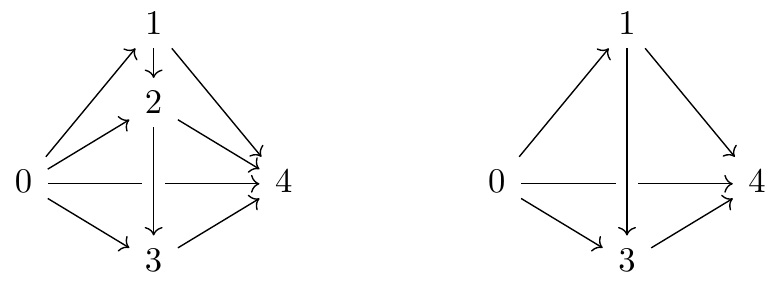}}
\]

It makes sense to interpret the right hand simplex, the 2nd face of the 2-admissible 4-simplex, as a composite of the 3rd and 1st faces because the 2-simplex
\[
\raisebox{-0.5\height}{\includegraphics{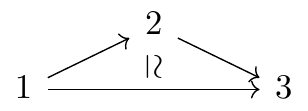}}
\]
is thin.
\end{example}

\subsection{$n$-trivialization and the $n$-core}\label{sec:n-trivial}

We  now introduce the complicial analog of the condition that an $(\infty,\infty)$-category is actually an $(\infty,n)$-category, in which each $r$-cell with $r > n$ is weakly invertible. 

\begin{definition} A stratified simplicial set $X$ is \textbf{$n$-trivial} if all $r$-simplices are marked for $r > n$. 
\end{definition}

The full subcategory of $n$-trivial stratified simplicial sets is reflective and coreflective
\[
\raisebox{-0.5\height}{\includegraphics{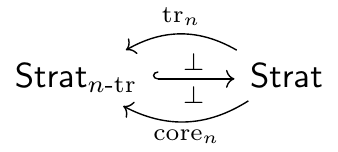}}
\]
in the category of stratified simplicial sets. That is $n$-\textbf{trivialization} defines an idempotent monad on $\cat{Strat}$ with unit the entire inclusion
\[ X \hookrightarrow_e \mathrm{tr}_nX\] of a stratified simplicial set $X$ into the stratified simplicial set $\mathrm{tr}_nX$ with the same marked simplices in dimensions $1,\ldots, n$, and with all higher simplices ``made thin.'' A complicial set is $n$-\textbf{trivial} if this map is an isomorphism.

The $n$-\textbf{core} $\mathrm{core}_nX$, defined by restricting to those simplices whose faces above dimension $n$ are all thin in $X$, defines an idempotent comonad with counit the regular inclusion
\[ \mathrm{core}_nX \hookrightarrow_r X.\] Again, a complicial set is $n$-trivial just when this map is an equivalence. As is always the case for a monad-comonad pair arising in this way, these functors are adjoints: $\mathrm{tr}_n \dashv \mathrm{core}_n$.

The subcategories of $n$-trivial stratified simplicial sets assemble to define a string of inclusions with adjoints
\[
\raisebox{-0.5\height}{\includegraphics{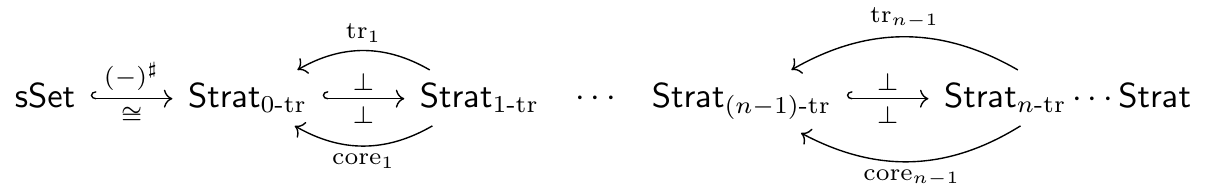}}
\]
that filter the inclusion of simplicial sets, considered as maximally marked stratified simplicial sets, into the category of all stratified simplicial sets.

\begin{exercise}
Show that the two right adjoints restrict to complicial sets to define functors that model the inclusion of $(\infty,n-1)$-categories into $(\infty,n)$-categories and its right adjoint, which takes an $(\infty,n)$-category to the ``groupoid core.'' 
\end{exercise}

\begin{remark} By contrast, the left adjoint, which just marks things arbitrarily, does not preserve complicial structure; this construction is too naive to define the ``freely invert $n$-arrows'' functor from $(\infty,n)$-categories to $(\infty,n-1)$-categories.\footnote{For instance, if $A$ is a naturally marked quasi-category, that's 1-trivial, then its zero trivialization is not a Kan complex (because we have not changed the underlying simplicial set) but it's groupoid core is (by a theorem of Joyal).}
\end{remark}

\begin{exercise} Show that a 0-trivial complicial set is exactly a Kan complex with the maximal ``$(-)^\sharp$'' marking.
\end{exercise}

\begin{exercise}
Prove that the underlying simplicial set of any 1-trivial complicial set is a quasi-category.
\end{exercise}

Conversely, any quasi-category admits a stratification making it a complicial set. The markings on the 1-simplices cannot be arbitrarily assigned. At minimum, certain automorphisms (endo-simplices that are homotopic to identities) must be marked. More to the point, each edge that is marked necessarily defines an equivalence in the quasi-category. But it is not necessary to mark all of the equivalences.

\begin{example} Strict $n$-categories define $n$-trivial \emph{strict} complicial sets, with unique fillers for the admissible horns, via the Street nerve, which is the subject of the next section.
\end{example}

In the third part of these notes, we argue that the complicial sets that most closely model $(\infty,n)$-categories are the $n$-trivial \emph{saturated} complicial sets, in which all \emph{equivalences} are marked.  In the case of an $n$-trivial stratification, the equivalences are canonically determined by the structure of the simplicial set.  One bit of evidence for the importance of the notion of saturation discussed below is the fact that the category of quasi-categories is isomorphic to the category of saturated 1-trivial complicial sets (Example \ref{ex:q-cat-natural}).

\section{The Street nerve of an $\omega$-category}\label{part:nerve}

The \emph{Street nerve} is a functor \[N \colon \vcat{\omega} \to \cat{sSet}\] from strict $\omega$-categories to simplicial sets. As is always the case for nerve constructions, the Street nerve  is determined by a functor
\[\mathcal{O}\colon \DDelta \to \vcat{\omega}.\] 
In this case, the image of $[n] \in \DDelta$ is the $n$th \emph{oriental} $\mathcal{O}_n$, a strict $n$-category  defined by Street \cite{street}. The nerve of a strict $\omega$-category $C$ is then defined to be the simplicial set whose $n$-simplices \[ NC_n :=\hom(\mathcal{O}_n,C)\] are $\omega$-functors $\mathcal{O}_n \to C$.  There are various ways to define a stratification on the nerve of an $\omega$-category, defining a lift of the Street nerve to a functor valued in stratified simplicial sets. One of these marking conventions turns Street nerves of strict $\omega$-categories into strict complicial sets, and indeed all strict complicial sets arise in this way. This is the content of the  Street--Roberts conjecture, proven by Verity, which motivated the definition of strict complicial sets.

\begin{theorem}[Verity]\label{thm:street-roberts}
The Street nerve defines a fully faithful embedding
\[ 
\raisebox{-0.5\height}{\includegraphics{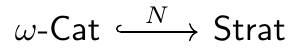}}
\] of $\omega$-categories into stratified simplicial sets, where an $n$-simplex $\mathcal{O}_n \to C$ in $NC$ is marked if and only if it carries the top dimensional $n$-cell on $\mathcal{O}_n$ to an identity in $C$. The essential image is the category of strict complicial sets.
\end{theorem}

In \S\ref{sec:omega-cat}, we introduce strict $\omega$-categories, and then in \S\ref{sec:oriental} we introduce the orientals. In \S\ref{sec:nerve}, we then define the Street nerve and revisit the Street--Roberts conjecture, though we leave the details of its proof to \cite{verity-complicial}. At the conclusion of this section, we look ahead to  \S\ref{sec:weak-from-strict}, which explores other marking conventions for Street nerves of strict $n$-categories. In this way, the Street nerve provides an important source of examples of weak, as well as strict, complicial sets. These are obtained by marking the equivalences and not just the identities in $NC$, the consideration of which  leads naturally to the notion of saturation in a complicial set, which is a main topic for the final section of these notes.

\subsection{$\omega$-categories}\label{sec:omega-cat}

Street's ``The algebra of oriented simplexes'' \cite{street} gives a single-sorted definition of a (strict) $n$-category in all dimensions $n=1,\ldots, \omega$. In the single-sorted definition of a 1-category, an object is identified with its identity morphism, and these 0-cells are recognized among the set of 1-cells as the fixed points for the source and target maps.

\begin{definition} A 1-\textbf{category} $(C,s,t,*)$ consists of 
\begin{itemize}
\item a set $C$ of \textbf{cells}
\item functions $s,t \colon C \rightrightarrows C$ so that $ss = ts =s$ and $tt = st = t$ (a target or source has itself as its target and its source).
\item a function $* \colon C \times_C C \to C$ from the pullback of $s$ along $t$ to $C$ so that $s(a * b) = s(b)$ and $t(a * b) = t(a)$ (the source of a composite is the source of its first cell and the target is the target of the second cell).
\end{itemize}
and so that 
\begin{itemize} 
\item $s(a) = t(v) = v$ implies $a *v = a$ (right identity)
\item $u = s(u) = t(a)$ implies $u * a = a$ (left identity)
\item $s(a) = t(b)$ and $s(b) = t(c)$ imply $a * (b * c) = (a * b) *c$ (associativity).
\end{itemize}
The \textbf{objects} or 0-\textbf{cells} are the fixed points for $s$ and then also for $t$ and conversely.
\end{definition}

\begin{definition} A 2-\textbf{category} $(C,s_0, t_0, *_0, s_1, t_1, *_1)$ consists of two 1-categories \[(C,s_0,t_0,*_0)\quad \mathrm{and} \quad (C,s_1, t_1, *_1)\] so that
\begin{itemize}
\item $s_1 s_0 = s_0 = s_0 s_1 = s_0 t_1$, $t_0 = t_0 s_1 = t_0 t_1$ (globularity plus 1-sources and 1-targets of points are points)
\item $s_0(a) = t_0(b)$ implies $s_1(a *_0 b ) = s_1(a) *_0 s_1 (b)$ and $t_1(a *_0 b) = t_1(a) *_0 t_1(b)$  (1-cell boundaries of horizontal composites are composites).
\item $s_1 (a) = t_1(b)$ and $s_1 (a') = t_1 (b')$ and $s_0(a) = t_0 (a')$ imply  that
\[ (a *_1 b) *_0 (a' *_1 b') = (a*_0 a') *_1 (b*_0 b')\] (middle four interchange).
\end{itemize}
Identities for $*_0$ are $0$-\textbf{cells} and identities for $*_1$ are $1$-\textbf{cells}.
\end{definition}

\begin{definition} An $\omega^+$-\textbf{category}\footnote{Street called these ``$\omega$-categories'' but we  reserve this term for something else.} consists of 1-categories $(C, s_n,t_n, *_n)$ for each $n \in \omega$ so that $(C, s_m, t_m, *_m, s_n, t_n, *_n)$ is a 2-category for each $m < n$. The identities for $*_n$ are $n$-\textbf{cells}.  An $\omega^+$-\textbf{functor} is a function that preserves sources, targets, and composition for each $n$.
\end{definition}

An $\omega$-\textbf{category} is an $\omega^+$-category in which every element is a \textbf{cell}, an $n$-cell for some $n$. Every $\omega^+$-category has a maximal sub $\omega$-category of cells and all of the constructions described here restrict to $\omega$-categories. 

An $n$-\textbf{category} is an $\omega$-category comprised of only $n$-cells. This means that the 1-category structures $(C, s_m, t_m, *_m)$ for $m > n$ are all discrete.

\begin{example} The underlying set functor $\omega^+$-$\cat{Cat} \to \cat{Set}$ is represented by the \textbf{free $\omega^+$-category $2_\omega$ on one generator}\footnote{In personal communication, Ross suggests that there may be something wrong with this example, but I do not see what it is.}, whose underlying set is
\[ (2 \times \omega) \cup \{ \omega\}.\]
The element $\omega$ is the unique non-cell, while the objects $(0,n)$ and $(1,n)$ are $n$-cells, respectively the $n$-source and $n$-target of $\omega$:
\[s_n(\omega) = (0,n)\quad\mathrm{and}\quad t_n(\omega) = (1,n).\] An $m$-cell is necessarily its own $n$-source and $n$-target for $m \leq n$; thus: 
\[s_n(\epsilon,m) = t_n(\epsilon,m) = (\epsilon,m)\quad \mathrm{for}\ m \leq n,\]
while:
\[s_n(\epsilon,m) = (0,n)\quad \mathrm{and}\quad t_n(\epsilon,m) = (1,n)\quad \mathrm{for}\ n < m.\]
 
The identity laws dictate all of the composition relations, e.g.: 
\[ \omega \ast_n (0,n) = \omega = (1,n) \ast_n \omega.\]
\end{example}

Using $2_\omega$ one can define the \textbf{functor $\omega^+$-category} $[A,B]$ for two $\omega^+$-categories $A$ and $B$: elements are $\omega^+$-functors $A \times 2_\omega \to B$. 

\begin{exercise}
Work out  the rest of the definition of the $\omega^+$-category $[A,B]$ and prove that $\omega^+$-$\cat{Cat}$ is cartesian closed. 
\end{exercise}

\begin{theorem}[Street]
There is an equivalence of categories
\[ \vcat{(\vcat{\omega^+})} \xrightarrow{\simeq} \vcat{\omega^+}\] which restricts to define an equivalence
\[ \vcat{(\vcat{n})} \xrightarrow{\simeq}	\vcat{(1+n)}\] for each $n \in [0,\omega]$.\footnote{Recall that in ordinal arithmetic $1+\omega  = \omega$.}
\end{theorem}
\begin{proof}
The construction of this functor is extends the construction of a 2-category from a $\cat{Cat}$-enriched category. Let $\cC$ be a category enriched in $\omega^+$-categories. Define an $\omega^+$-category $C$ whose underlying set is
\[ C := \coprod_{u,v \in \ob\cC} \cC(u,v).\] The $0$-source and $0$-target of an element $a \in \cC(u,v)$ are $u$ and $v$, respectively, and $0$-composition is defined using the enriched category composition. The $n$-source,$n$-target, and $n$-composition are defined using the $(n-1)$-category structure of the $\omega^+$-category $\cC(u,v)$.

Conversely, the $\omega^+$-category enriched category of an $\omega^+$-category $C$ can be defined by taking its 0-cells as its objects, defining $C(u,v)$ to be the collection of elements with $0$-source $u$ and $0$-target $v$, using the operations $(s_n,t_n,*_n)$ for $n > 0$ to define the $\omega^+$-category structure on $C(u,v)$.
\end{proof}

\subsection{Orientals}\label{sec:oriental}

The $n$th \emph{oriental} $\mathcal{O}_n$ is a strict $n$-category  with a single $n$-cell whose source is the pasted composite of $(n-1)$-cells, one for each of the odd faces of the simplex $\Delta[n]$, and whose target is a pasted composite of $(n-1)$-cells, one for each of the even faces of the simplex $\Delta[n]$. The orientals $\mathcal{O}_n$ can be recognized as full sub $\omega$-categories of an $\omega$-category $\mathcal{O}_\omega$,  the free $\omega$-category on the $\omega$-simplex $\Delta[\omega]$, spanned by the objects that correspond to the vertices of $\Delta[n]$. The precise combinatorial definition of $\mathcal{O}_n$ is rather subtle to state, making use of Street's notion of \emph{parity complex}, which we decline to introduce in general. Before defining the orientals as special cases of parity complexes, we first describe the low-dimensional cases.

The orientals $\mathcal{O}_0, \mathcal{O}_1, \mathcal{O}_2,\ldots$ are $\omega$-categories, where each $\mathcal{O}_n$ is an $n$-category. In low dimensions:
\begin{itemize}
\item[($n=0$)] $\mathcal{O}_0$ is the $\omega$-category with a single $0$-cell:
\[ 0 \]
\item [($n=1$)]  $\mathcal{O}_1$ is the $\omega$-category with two 0-cells $0,1$ and a 1-cell:
\[ 
\raisebox{-0.5\height}{\includegraphics{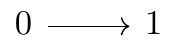}}
\]
\item[ ($n=2$)] $\mathcal{O}_2$ is the $\omega$-category with three 0-cells $0,1,2$ and four 1-cells as displayed:
\[
\raisebox{-0.5\height}{\includegraphics{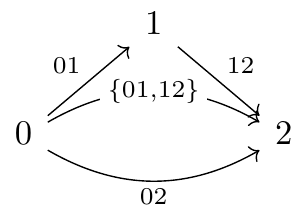}}
\]
Note that only  two of these are composable, with their composite the 1-cell denoted by $\{01,12\}$. The underlying 1-category of $\mathcal{O}_2$ is the non-commutative triangle, the free 1-category  generated by the ordinal $[2]$.

There is a unique 2-cell 
\[
\raisebox{-0.5\height}{\includegraphics{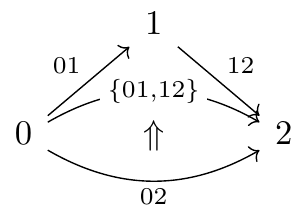}}
\]
whose 0-source is 0 and whose 0-target is 2, and whose 1-source is 02 and whose 1-target is $\{01,12\}$.
We can simplify our pictures by declining to draw the free composites that are present in 
 $\mathcal{O}_2$, as they must be in any $\omega$-category. Under this simplifying convention, $\mathcal{O}_2$ is depicted as:
\[
\raisebox{-0.5\height}{\includegraphics{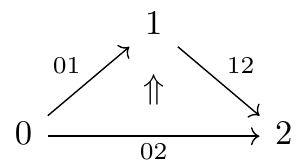}}
\]
\item[($n=3$)] Similarly $\mathcal{O}_3$ has four 0-cells, abbreviated 0,1,2,3; has the free category on the graph [3] as its underlying 1-category, with six atomic 1-cells and five free composites; has four atomic 2-cells plus two composites; and has a 3-cell from one of these composites to the other. Under the simplifying conventions established above, $\mathcal{O}_3$ can be drawn as:
\[
\raisebox{-0.5\height}{\includegraphics{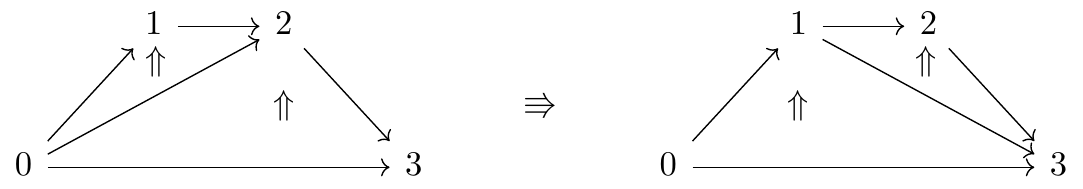}}
\]
\end{itemize}

\begin{definition}[the $n$th oriental, informally] The $n$th \textbf{oriental} is the strict $n$-category $\mathcal{O}_n$ whose 
 atomic $k$-cells corresponding to the $k$-dimensional faces of $\Delta[n]$ (the non-degenerate $k$-simplices, which can be identified with $(k+1)$-element subsets of $[n]$).  The co\-dimension-one source of a $k$-cell is a pasted composite of the odd faces of the $\Delta[k]$-simplex, while the codimension-one target is a pasted composite of the even faces of the $k$-simplex. If $S$ is a subset of faces of $\Delta[k]$ write $S^{-}$ for the union of the odd faces of simplices in $S$ and write $S^+$ for the union of even faces of simplices in $S$. Write $S_k$ for the $k$-dimensional elements of $S$ and $|S|_k$ for the elements of dimension at most $k$.
\end{definition}

\begin{definition}[the $n$th oriental, precisely]
The $k$-cells of the $n$-category $\mathcal{O}_n$ are pairs $(M,P)$ where $M$ and $P$ are non-empty, \textbf{well-formed}, finite subsets of faces of $\Delta[n]$ of dimension at most $k$ so that $M$ and $P$ both \textbf{move $M$ to $P$}. Here a subset $S$ of faces of $\Delta[n]$ is \textbf{well-formed} if it contains at most one vertex and if for any distinct elements $x\neq y$, $x$ and $y$ have no common sources and no common targets. A subset $S$ \textbf{moves $M$ to $P$} if
\[ M = (P \cup S^-) \backslash S^+\qquad \mathrm{and} \qquad P = (M \cup S^+) \backslash S^-.\]

If $(M,P)$ is a $m$-cell, the axioms imply that $M_m = P_m$. The $k$-source and $k$-target are given by
\begin{align*}
s_k(M,P) &:= (|M|_k, M_k \cup |P|_{k-1}) \\
t_k(M,P) &:= (|M|_{k-1} \cup P_k, |P|_k)
\end{align*}
and composition is defined by
\[ (M,P) *_k (N,Q) := ( M \cup (N \backslash N_k), (P \backslash P_k) \cup Q).\]
\end{definition}

\begin{example} The oriental $\mathcal{O}_4$ has a unique 4-cell given by the pair
\begin{align*} M &= \{ 01234, 0124,0234, 012,023,034,04,0\}\\  P &= \{01234,0123,0134,1234,124,234,014,01,12,23,34,4\}.\end{align*}
\end{example}

\begin{exercise} Identify the source and target of the unique 4-cell in $\mathcal{O}_4$.
\end{exercise}

\begin{exercise} Show that $\mathcal{O}_n$ has a unique $n$-cell.
\end{exercise}

The orientals satisfy the universal property of being freely generated by the faces of the simplex, in the sense of the following definition of free generation for an $\omega$-category.

\begin{definition} For an $\omega$-category $C$, write $\vert{C}\vert_n$ for its $n$-categorical truncation, discarding all higher-dimensional cells. The $\omega$-category $C$ is \textbf{freely generated} by a subset $G \subset C$ if for each $\omega$-category $X$, $n \in \omega$, $n$-functor $\vert{C}\vert_n \to X$, and map $G \cap \vert{C}\vert_{n+1} \to X$, compatible with $n$-sources and targets there exists a unique extension to an $(n+1)$-functor $\vert{C}\vert_{n+1} \to X$.
\[
\raisebox{-0.5\height}{\includegraphics{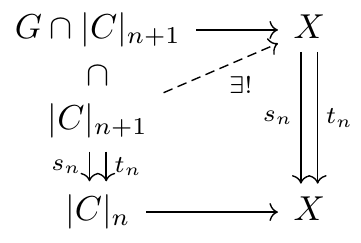}}
\]
\end{definition}

\begin{theorem}[Street] The category $\mathcal{O}_n$ is freely generated by the faces of $\Delta[n]$.
\end{theorem}

\begin{exercise}
Use this universal property  to show that the orientals define a cosimplicial object in $\omega$-categories
\[ \mathcal{O} \colon \DDelta \to \vcat{\omega}.\]
\end{exercise}

This cosimplicial object gives rise to the Street nerve, to which we now turn.

\subsection{The Street nerve as a strict complicial set}\label{sec:nerve}

\begin{definition}
The \textbf{Street nerve} of an $\omega$-category $C$,  is the simplicial set
$NC$ whose $n$-simplices are $\omega$-functors $\mathcal{O}_n \to C$.
\end{definition}

\begin{example}[Street nerves of low-dimensional categories] $\quad$
\begin{enumerate}
\item The Street nerve of a 1-category is its usual nerve.
\item  The Street nerve of a 2-category has 0-simplices the objects, 1-simplices the 1-cells, and 2-simplices the 2-cells $\alpha\colon h \To gf$ whose target is a specified composite
\[
\raisebox{-0.5\height}{\includegraphics{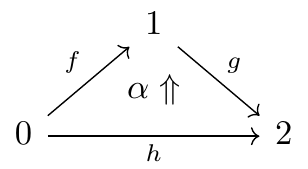}}
\]
The 3-simplices record equations between pasted composites of 2-cells of the form
\[
\raisebox{-0.5\height}{\includegraphics{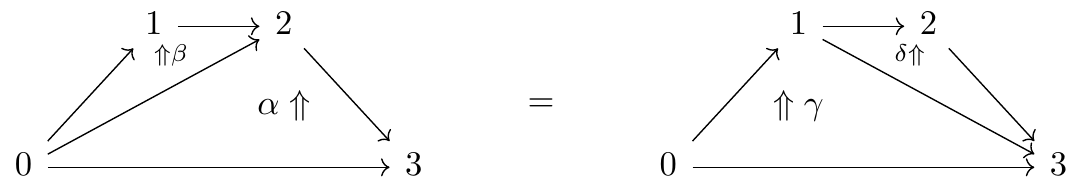}}
\]
This simplicial set is 3-coskeletal, with a unique filler for all spheres in higher dimensions.
\end{enumerate}
\end{example}

In general:

\begin{theorem}[Street] The nerve of an $n$-category is $(n+1)$-coskeletal.
\end{theorem}

The Street nerve can be lifted along $U \colon \cat{Strat} \to \cat{sSet}$ by choosing a stratification for the simplicial set $NC$. 

\begin{definition} In the \textbf{identity stratification} of the Street nerve of an $\omega$-category $C$, an $n$-simplex in $NC$ is marked if and only if the corresponding $\omega$-functor $\mathcal{O}_n \to C$ carries the $n$-cell in $\mathcal{O}_n$ to a cell of lower dimension in $C$. That is, in the identity stratification of $NC$, only those $n$-simplices corresponding to identities are marked.
\end{definition}

The identity stratification defines a functor $\vcat{\omega} \to \cat{Strat}$. This terminology allows us to restate the Street--Roberts conjecture more concisely:

\begin{theorem}[Verity]
The Street nerve with the identity stratification defines a fully faithful embedding 
\[ 
\raisebox{-0.5\height}{\includegraphics{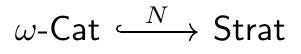}}
\] of $\omega$-categories into stratified simplicial sets, with essential image the category of strict complicial sets.
\end{theorem}

\begin{example}$\quad$
\begin{enumerate}
\item If $C$ is a 1-category, the identity stratification turns $NC$ into a 2-trivial strict complicial set with only the identity (i.e., degenerate) 1-simplices marked. 
\item If $C$ is a 2-category, the identity stratification  turns $NC$ into a 3-trivial strict complicial set with only the degenerate 1-simplices marked and with a 2-simplex marked if and only if it is inhabited by an identity 2-cell, whether or not there are degenerate edges, e.g.,:
\[
\raisebox{-0.5\height}{\includegraphics{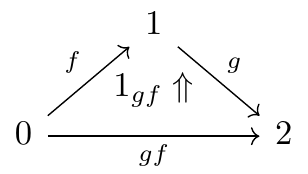}}
\]
\end{enumerate}
\end{example}

An interesting feature of the complicial sets model of higher categories is that strict $\omega$-categories can also be a source of  \emph{weak} rather than \emph{strict} complicial sets, simply by choosing a more expansive marking convention. We begin the next section by exploring this possibility.

\section{Saturated complicial sets}\label{part:saturation}

In the previous section, we defined the Street nerve of an $\omega$-category $C$, a simplicial set $NC$ whose $n$-simplices are diagrams $\mathcal{O}_n \to C$ indexed by the $n$th oriental. We observed that this simplicial set becomes a strict complicial set if we mark precisely those diagrams $\mathcal{O}_n \to C$ that carry the $n$-cell of $\mathcal{O}_n$ to a cell of dimension less than $n$ in $C$ (i.e., to an identity).

One of the virtues of the complicial sets model of weak higher categories is the possibility of changing the stratification on a given simplicial set if one desires a more generous or more refined notion of thinness, corresponding to a tighter or looser definition of composition. The identity stratification of $NC$ is the smallest stratification that makes this simplicial set into a weak complicial set, but we will soon meet other larger stratifications that are more categorically natural.

In \S\ref{sec:weak-from-strict}, we begin by looking in low dimensions for limitations on which simplices can be marked in a complicial set, and discover that any marked 1-simplex is necessarily an 1-equivalence, in a sense that we define. In \S\ref{sec:saturation}, we introduce the higher-dimensional generalization of these notions. We conclude in \S\ref{sec:model} by summarizing the work of Verity that establishes the basic homotopy theory of complicial sets of various flavors.

The construct weak complicial sets from nerves of strict $\omega$-categories, the stratification on the Street nerve is enlarged,  but in other instances refinement of the markings is desired. For example, Verity constructs a Kan complex of simplicial cobordisms between piecewise-linear manifolds. Because the underlying simplicial set is a Kan complex, it becomes a weak complicial set under the 0-trivial stratification where all cobordisms (all positive-dimensional simplices) are marked. Other choices, in increasing order of refinement, are to mark the $h$-cobordisms (cobordisms for which the negative and positive boundary inclusions are homotopy equivalences), the quasi-invertible cobordisms (the ``equivalences''), or merely the trivial cobordisms (meaning the cobordism ``collapses'' onto its negative and also its positive boundary).

\subsection{Weak complicial sets from strict $\omega$-categories}\label{sec:weak-from-strict}

To explore other potential markings of Street nerves of strict $\omega$-categories, we first ask whether it is possible to mark more than just the degenerate 1-simplices?

If $f$ is a marked edge in any complicial set $A$, then the $\Lambda^2[2]$-horn with 0th face $f$ and 1st face degenerate is admissible, so $f$ has a right equivalence inverse. A dual construction involving a $\Lambda^0[2]$-horn shows that $f$ has a left equivalence inverse.
\begin{equation}\label{eq:1-equiv}
\raisebox{-0.5\height}{\includegraphics{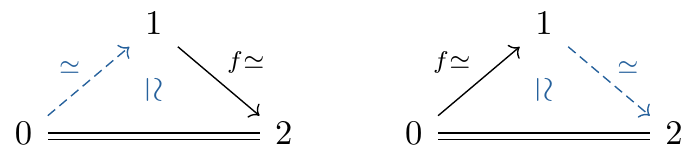}}
\end{equation}
 The elementary thinness extensions imply further than these one-sided inverses are also marked, so they  admit further inverses of their own. 
 
  \begin{definition}\label{defn:1-equivalence} A 1-simplex in a stratified simplicial set is a 1-\textbf{equivalence} if there exist a pair of thin 2-simplices as displayed
 \[
\raisebox{-0.5\height}{\includegraphics{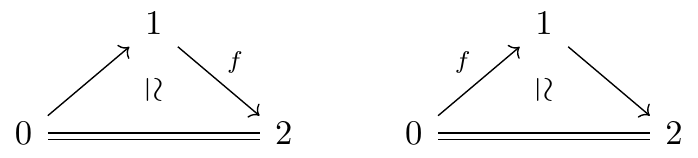}}
\]
\end{definition}

Note the notion of 1-equivalence is defined relative to the 2-dimensional stratification.  

\begin{remark}
There are many equivalent ways to characterize the 1-equivalences in a complicial set $A$. We choose Definition \ref{defn:1-equivalence} because of its simplicity and naturality, and because this definition provides a homotopically well-behaved type of equivalences in homotopy type theory; see \cite[2.4.10]{HoTT}.
\end{remark}

The elementary anodyne extensions displayed in \ref{eq:1-equiv} prove:
 
 \begin{proposition}\label{prop:marked-edge-equivalences} Any marked 1-simplex in a complicial set is a 1-equivalence.
 \end{proposition}
 
 This result suggests an alternate stratification for nerves of 1-categories:
 
 \begin{proposition} If $C$ is a 1-category then the 2-trivial stratification of $NC$ with the isomorphisms as marked 1-simplices defines a complicial set.
 \end{proposition}

Depending on the 1-category there may be intermediate stratifications where only some of the isomorphisms are marked (the set of marked edges has to satisfy the 2-of-3 property) but these are somehow less interesting.

 \begin{exercise}
 Prove this.
 \end{exercise}

Let us now consider the degenerate edges, the thing edges, and the 1-equivalences as subsets of  the set of 1-simplices in a complicial set $A$. In any stratified simplicial set, the degenerate 1-simplices are necessarily thin. In a complicial set $A$, Proposition \ref{prop:marked-edge-equivalences} proves that  the thin 1-simplices are necessarily 1-equivalences, but there is nothing in the complicial set axioms that guarantees that all equivalences are marked. We introduce terminology that characterizes when this is the case:

\begin{definition}\label{defn:1-sat}
A complicial set $A$ is 1-\textbf{saturated} if every 1-equivalence is marked.
\end{definition}

If a 1-trivial complicial set is 1-saturated then it is \emph{saturated} in the sense of Definition \ref{defn:saturation} below. From the definitions, it is easy to prove:

\begin{proposition}\label{prop:1-cat-stratified}
If $C$ is a strict 1-category, there is a unique saturated 1-trivial complicial structure on $NC$, namely the one in which every isomorphism in $C$ is marked. Moreover, this is the maximal 1-trivial stratification making $NC$ into a complicial set. 
\end{proposition}

\begin{exercise} Prove this.
\end{exercise}

To build intuition for higher dimensional generalizations of these notions, next consider the Street nerve of a strict 2-category as a 2-trivial stratified simplicial set. As the notion of 1-saturation introduced in Definitions \ref{defn:1-equivalence} and \ref{defn:1-sat} depends on the markings of 2-simplices, it makes sense to consider the markings on the 2-simplices  first. If only identity 2-simplices are marked, then the 1-saturation of $NC$ is as before: marking all of the 1-cell isomorphisms in the 2-category $C$. But we might ask again whether a larger stratification is possible at level 2.

In any complicial set, consider a thin 2-simplex $\alpha$ with 0th edge degenerate. From $\alpha$ one can build admissible $\Lambda^1[3]$ and admissible $\Lambda^3[3]$-horns admitting thin fillers:
\[
\raisebox{-0.5\height}{\includegraphics{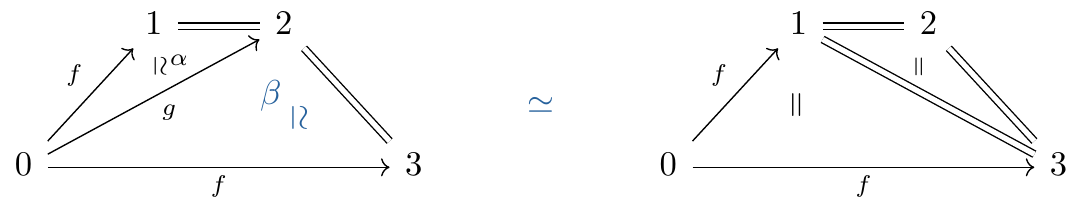}}
\]
\[
\raisebox{-0.5\height}{\includegraphics{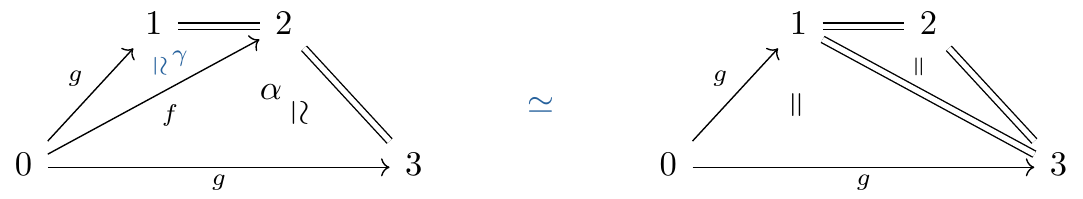}}
\]
So again we conclude that any thin 2-simplex of this form is necessarily an ``equivalence'' up to thin 3-simplices, in the sense of the displayed diagrams. Informally, a complicial set is 2-\emph{saturated} if all 2-simplices that are equivalences in this sense are marked. A precise definition of saturation that applies in any dimension appears momentarily as Definition \ref{defn:saturation}. It follows that:

\begin{proposition}\label{prop:2-cat-stratified}
If $C$ is a strict 2-category, there is a unique saturated 2-trivial complicial structure on $NC$, in which the 2-cell isomorphisms and the 1-cell equivalences are marked. Moreover, this is the maximal 2-trivial stratification making $NC$ into a complicial set. 
\end{proposition}

Unlike the 1-trivial saturated stratification on the Street nerve of a 1-category described in Proposition \ref{prop:1-cat-stratified}, the 2-trivial saturated stratification on the Street nerve of a 2-category described in Proposition \ref{prop:2-cat-stratified} describes a \emph{weak} and not a \emph{strict} complicial set.

\subsection{Saturation}\label{sec:saturation}

To define saturation in any dimension, it  is convenient to rephrase the definition of 1-saturation as a lifting property. The pair of thin 2-simplices 
 \[
\raisebox{-0.5\height}{\includegraphics{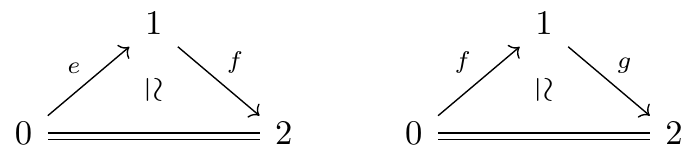}}
\]
define the 3rd and 0th faces of an inner admissible $\Lambda^1[3]$- or $\Lambda^2[3]$-horn that fills to define a thin 3-simplex
\[
\raisebox{-0.5\height}{\includegraphics{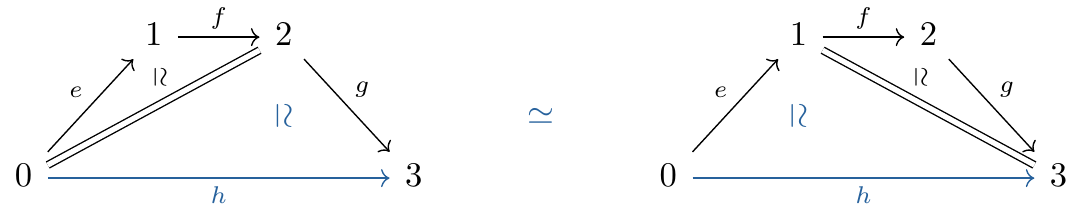}}
\]
This 3-simplex defines a map $\Delta[3]_{\mr{eq}} \to A$, where $\Delta[3]_{\mr{eq}}$ is the 3-simplex given a 1-trivial stratification with the edges $[02]$ and $[13]$ also marked. 

\begin{proposition}\label{prop:1-saturated-lifting} A complicial set $A$ is 1-saturated if and only if it admits extensions along the entire inclusion of $\Delta[3]_{\mr{eq}}$ into the maximally marked 3-simplex:
\[
\raisebox{-0.5\height}{\includegraphics{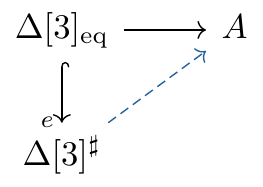}}
\]
\end{proposition}

\begin{exercise}
Prove this.
\end{exercise}

There are similar extension problems that detects saturation in any dimension, which are defined by forming the join of the inclusion $\Delta[3]_{\mr{eq}}\hookrightarrow_e \Delta[3]^\sharp$ with simplices on one side or the other. 

\begin{definition}[join and slice]
The ordinal sum on $\DDelta_+$ extends via Day convolution to a bifunctor on the category of augmented simplicial sets called the \textbf{join}. Any simplicial set can be regarded as a trivially augmented simplicial set. Under this inclusion, the join restricts to define a bifunctor
\[ \cat{sSet} \times \cat{sSet} \xrightarrow{\star} \cat{sSet}\] so that $\Delta[n] \star \Delta[m] = \Delta[n+m+1]$. 
More generally, an $n$-simplex in the join $A \star B$ of two simplicial sets is a pair of simplices $\Delta[k] \to A$ and $\Delta[n-k-1] \to B$ for some $-1 \leq k < n$. Here $\Delta[-1]$ is the trivial augmentation of the empty simplicial set, in which case the functors $\Delta[-1]\star -$ and $-\star\Delta[-1]$ are naturally isomorphic to the identity.

The \textbf{left} and \textbf{right slices} of a simplicial set $A$ over a simplex $\sigma \colon \Delta[n] \to A$ are the simplicial sets $\sigma\backslash A$ and $A/\sigma$ whose $k$-simplices correspond to diagrams
\begin{equation}\label{eq:join-slice}
\raisebox{-0.5\height}{\includegraphics{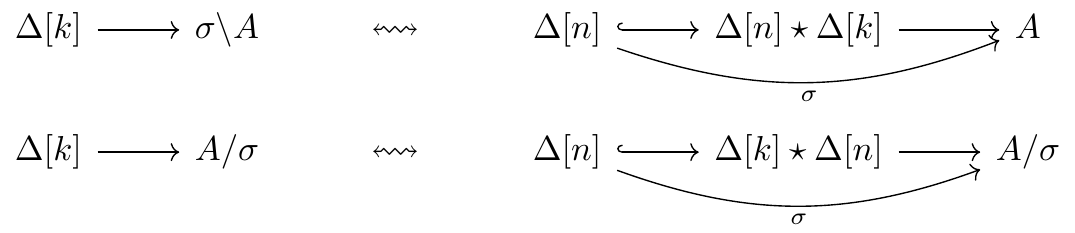}}
\end{equation}
See \cite{joyal} for more.
\end{definition}

\begin{definition}[stratified join]
The simplicial join lifts to a join bifunctor
\[ \cat{Strat} \times \cat{Strat} \xrightarrow{\star} \cat{Strat}\] in which a simplex $\Delta[n] \to A \star B$, with components $\Delta[k] \to A$ and $\Delta[n-k-1] \to B$, is marked in $A\star B$ if and only if at least one of the simplices in $A$ or $B$ is marked. More details can be found in \cite{verity-complicial}.
\end{definition}

\begin{exercise}
Define a stratification on the slices $\sigma\backslash A$ and $A/\sigma$ over an $n$-simplex $\sigma \colon \Delta[n] \to A$ so that the correspondence \eqref{eq:join-slice} extends to stratified simplicial sets.
\end{exercise}

\begin{definition}\label{defn:saturation} A complicial set is \textbf{saturated} if it admits extensions along the set of entire inclusions
\[ \{ \Delta[m] \star \Delta[3]_{\mathrm{eq}} \star \Delta[n] \hookrightarrow_e \Delta[m] \star \Delta[3]^\sharp \star \Delta[n] \mid n,m \geq -1\}.\]
In fact, it suffices to require only extensions 
\[
\raisebox{-0.5\height}{\includegraphics{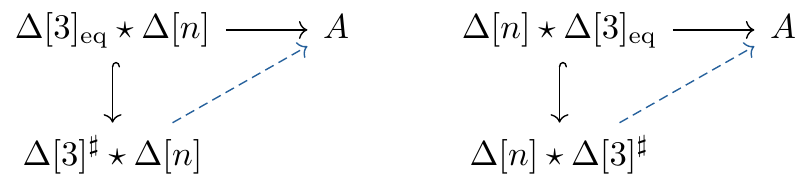}}
\]
along inclusions of one-sided joins of the inclusion $\Delta[3]_{\mathrm{eq}}\hookrightarrow_e\Delta[3]^\sharp$ with an $n$-simplex for each $n  \geq -1$, and as it turns out only the left-handed joins or right-handed joins are needed.
\end{definition}

By Proposition \ref{prop:1-saturated-lifting}, the $n=-1$ case of Definition \ref{defn:saturation}  asserts that every 1-equivalence in $A$, defined relative to the  marked 2-simplices and marked 3-simplices, is marked. By Proposition \ref{prop:1-saturated-lifting} again, the general extension property
\[ 
\raisebox{-0.5\height}{\includegraphics{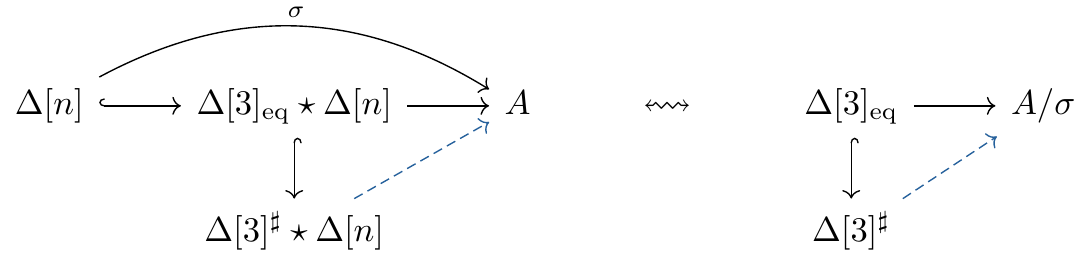}}
\]
 asserts that every 1-equivalence in the slice complicial set $A/\sigma$ is marked. 

At first blush, Definition \ref{defn:saturation} does not seem to be general enough. In the case of a vertex $\sigma \colon \Delta[0] \to A$, 1-equivalences in $A/\sigma$ define 2-simplices in $A$ whose $[01]$-edge is a 1-equivalence.  In particular, a generic 2-simplex
\[
\raisebox{-0.5\height}{\includegraphics{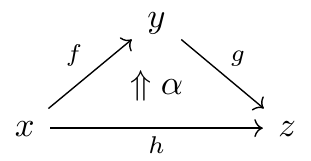}}
\]
with no 1-equivalence edges along its boundary, does not define a 1-equivalence in any slice complicial set. However, there are admissible 3-horns that can be filled to define the pasted composites of $\alpha$ with $1_f$ and $1_g$, respectively:
\[
\raisebox{-0.5\height}{\includegraphics{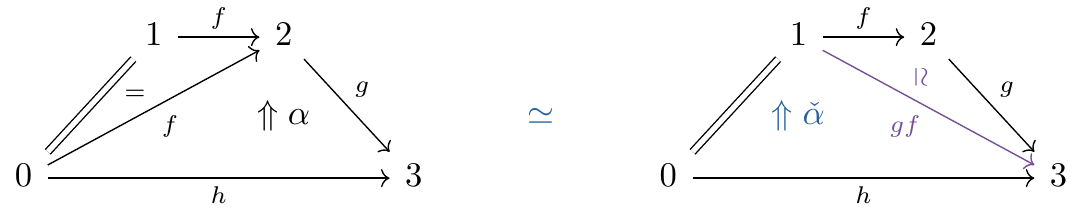}}
\]

\[
\raisebox{-0.5\height}{\includegraphics{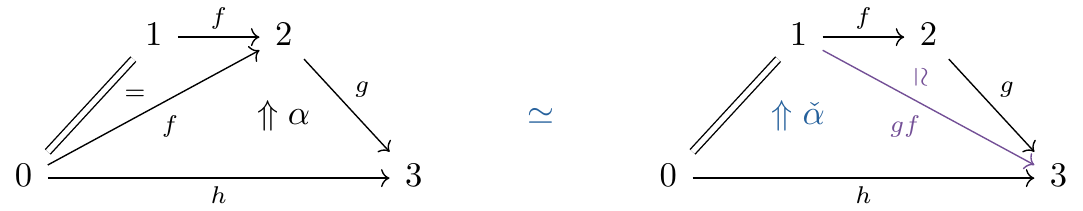}}
\]
By the complicial thinness extension property, if any of $\alpha$, $\hat{\alpha}$, or $\check{\alpha}$ are marked, then all of them are. 

\begin{exercise}
Generalize this  ``translation'' argument to prove that any $n$-simplex in a complicial set is connected via a finite sequence of $n$-simplices to an $n$-simplex whose first face is degenerate and an $n$-simplex whose last face is degenerate in such a way that if any one of these simplices is thin, they all are.
\end{exercise}

\begin{definition}\label{defn:n-equiv} In an $n$-trivial complicial set, an $n$-simplex $\sigma \colon \Delta[n] \to A$ is an $n$-\textbf{equivalence} if it admits an extension
\[
\raisebox{-0.5\height}{\includegraphics{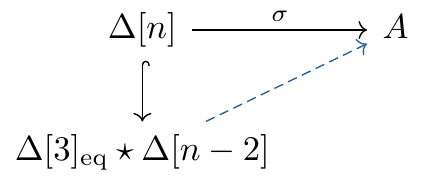}}
\]
along the map $\Delta[n] \hookrightarrow \Delta[3]_{\mr{eq}}\star \Delta[n-2]$ whose image includes the edge $[1,2]$ of $\Delta[3]_{\mr{eq}}$ and all of the vertices of $\Delta[n-2]$.
\end{definition}

\begin{remark}
The set of $n$-equivalences identified by Definition \ref{defn:n-equiv} depends on the marked $(n+1)$-simplices, which is the reason we have only stated this definition for an $n$-trivial complicial set. The $n$-equivalences in a generic complicial set are characterized by an inductive definition, the formulation of which we leave to the reader.
\end{remark}

\begin{example}[quasi-categories as complicial sets]\label{ex:q-cat-natural}
Expanding on the work of \S\ref{sec:quasi}, a quasi-category has a unique saturated stratification making it a complicial set: namely the 1-trivial saturation where all of the 1-equivalences are marked. This is the ``natural marking'' discussed in \cite{lurie}. Conversely, any 1-trivial saturated complicial set is a quasi-category. So quasi-categories are precisely the 1-trivial saturated complicial sets.
\end{example}

Each simplicial set has a minimum stratification, with only degeneracies marked. Because the definition of saturation is inductive, each simplicial set also has a minimum saturated stratification. Larger saturated stratifications also exist (e.g., the maximal marking of all positive-dimensional simplices). It is more delicate to describe how the process of saturating a given  complicial set interacts with the  complicial structure: adding new thin simplices adds new admissible horns which need fillers. What is more easily understood are model structures whose fibrant objects are complicial sets of a particular form, a subject to which we now turn.

\subsection{Model categories of complicial sets}\label{sec:model}

The category of stratified simplicial sets is cartesian closed, where the cartesian product $\times$ behaves like the \emph{Gray tensor product} in higher category theory.\footnote{Note that in the theory of bicategories, the cartesian product plays the role of the Gray tensor product in 2-category theory, in the sense that there is a biadjunction between the cartesian product and the hom-bicategory of pseudofunctors, pseudo-natural transformations, and modifications.} We write ``$\hom$'' for the internal hom characterized by the 2-variable adjunction
\[ \cat{Strat}(A \times B,C) \cong \cat{Strat}(A,\hom(B,C)) \cong \cat{Strat}(B,\hom(A,C)).\]

Let 
\[ I := \{ \partial\Delta[n]\hookrightarrow \Delta[n] \mid n \geq 0\} \cup \{ \Delta[n] \hookrightarrow\Delta[n]_t \mid n \geq 0 \}\] denote the generating set of monomorphisms of stratified simplicial sets introduced in Proposition \ref{prop:monos} and let 
\[J := \{ \Lambda^k[n] \hookrightarrow_r \Delta^k[n] \mid n \geq 1, k \in [n] \} \cup \{ \Delta^k[n]' \hookrightarrow_e \Delta^k[n]'' \mid n \geq 2, k \in [n] \}\] denote the set of \textbf{elementary anodyne extensions} introduced in Definition \ref{defn:complicial}, the right lifting property against which characterizes the complicial sets. A combinatorial lemma proves that the pushout product $I \hat{\times} J$ of maps in $I$ with maps in $J$ is an \textbf{anodyne extension}: that is, may be expressed as a retract of a transfinite composite of pushouts of coproducts of elements of $J$ (here mere composites of pushouts suffice). As a corollary:


\begin{proposition}[{Verity \cite{verity-weak-i}}]  If $X$ is a stratified simplicial set and $A$ is a weak complicial set, then $\hom(X,A)$ is a weak complicial set.
\end{proposition}

Verity provides a very general result for constructing model structures whose fibrant objects are  defined relative to some set of monomorphisms $K$ containing $J$. Call a stratified simplicial set a $K$-\textbf{complicial set} if it  admits extensions along each map in $K$. Suppose $K$ is a set of monomorphisms of $\cat{Strat}$ so that
\begin{enumerate}
\item every elementary anodyne extension is in $K$
\end{enumerate}
and moreover each of/all of the following equivalent conditions hold for each $j \in K$:
\begin{enumerate}[resume]
\item Each element $j$ of $K$ is a $K$-\textbf{weak equivalence}: i.e., $\hom(j,A)$ is a homotopy equivalence\footnote{Two maps $f,g \colon X \to A$ are \textbf{homotopic} if they extend to a map $X \times \Delta[1]^\sharp \to A$. If $A$ is a weak complicial set, this ``simple homotopy'' is an equivalence relation.} for each $K$-fibrant stratified set.
\item $\hom(j,A)$ is a trivial fibration for each $K$-complicial set.
\item Each $K$-complicial set admits extensions along all the maps $i \hat{\times} j$ for all $i \in I$ and $j \in K$.
\end{enumerate}
Call a map that has the right lifting property with respect to the set $K$ a $K$-\textbf{complicial fibration}.

\begin{theorem}[Verity \cite{verity-weak-i}]\label{thm:model} Each set of stratified inclusions $K$ satisfying the conditions (i)--(iv) gives rise to a cofibrantly generated model structure whose:
\begin{itemize}
\item weak equivalences are the $K$-weak equivalences,
\item cofibrations are monomorphisms,
\item fibrant objects are the $K$-complicial sets, and 
\item fibrations between fibrant objects are $K$-complicial fibrations.
\end{itemize}
Moreover, such a model structure is monoidal with respect to the Gray tensor product.
\end{theorem}
\begin{proof}
Apply Jeff Smith's theorem \cite[125]{verity-weak-i}.
\end{proof}


\begin{example} Theorem \ref{thm:model} applies to the minimal set of elementary anodyne extensions
\[ J := \{ \Lambda^k[n] \hookrightarrow_r \Delta^k[n] \mid n \geq 1, k \in [n] \} \cup \{ \Delta^k[n]' \hookrightarrow_e \Delta^k[n]'' \mid n \geq 2, k \in [n] \}\] defining the \textbf{model structure for complicial sets}.
\end{example}

\begin{example}
Theorem \ref{thm:model} applies to the union of the minimal $J$ with 
 \[K^{\mathrm{tr}}_n := \{ \Delta[r] \hookrightarrow_e \Delta[r]_t \mid r > n\}\] defining the \textbf{model structure for $n$-trivial complicial sets}. 
\end{example}

\begin{example}
Theorem \ref{thm:model} applies to the union of the minimal $J$ with 
 \[ K^\mathrm{s} := \{ \Delta[m] \star \Delta[3]_{\mr{eq}} \star \Delta[n] \hookrightarrow \Delta[m] \star \Delta[3]^\sharp \star \Delta[n] \mid m,n \geq -1\}\] defining the \textbf{model structure for saturated complicial sets}.
\end{example}

\begin{example}
Theorem \ref{thm:model} applies to the union of the minimal $J$ with 
both $K^{\mathrm{tr}}_n$ and $K^\mathrm{s}$ defining the \textbf{model structure for $n$-trivial saturated complicial sets}.
\end{example}

By Example \ref{ex:q-cat-natural}, the $n=1$ case of this last result gives a new proof of Joyal's model structure for quasi-categories.

\end{document}